\documentclass[11pt,a4paper]{article}

\usepackage[left=1.1in, right=1.1in, top=1.0in, bottom=1.0in]{geometry}
\usepackage{amsfonts,amstext,amsmath,amssymb,amsopn,amsthm}
\usepackage{mathrsfs}
\usepackage{mathtools}
\usepackage{physics}
\usepackage[ruled,vlined]{algorithm2e}
\usepackage{todonotes}
\usepackage{enumerate}

\usepackage[colorlinks=true, 
linkcolor=blue,
pdfstartview=FitH,      
breaklinks=true,        
bookmarksopen=true,     
bookmarksnumbered=true  
]{hyperref}
\usepackage[capitalize,nameinlink]{cleveref}

\theoremstyle{plain}
\newtheorem{thm}{Theorem}

\newtheorem{proposition}[thm]{Proposition}
\newtheorem{lemma}[thm]{Lemma}
\newtheorem{corollary}[thm]{Corollary}
\newtheorem{mydef}[thm]{Definition}
\newtheorem{remark}[thm]{Remark}

\newcommand{\paren}[1]{\left(#1\right)}
\newcommand{\mymap}[3]{#1:\,#2 \to #3\,}
\newcommand{\mmap}[3]{#1:\,#2\rightrightarrows #3\,}

\DeclareMathOperator{\argmin}{argmin\,}

\DeclareMathOperator{\prox}{prox}

\DeclareMathOperator{\Id}{Id}
\DeclareMathOperator{\Fix}{Fix}
\DeclareMathOperator{\dist}{dist}

\newcommand{\Nbb}{\mathbb{N}}

\newcommand{\Rbb}{\mathbb{R}}

\newcommand{\Ball}{\mathbb{B}}
\newcommand{\Tbar}{\overline{T}}
\newcommand{\Tcal}{\mathcal{T}}
\newcommand{\xbar}{\overline{x}}
\newcommand{\tbar}{\overline{t}}
\newcommand{\ybar}{{\overline{y}}}

\newcommand{\alphabar}{\overline{\alpha}}
\newcommand{\alphahat}{\widehat{\alpha}}
\newcommand{\epsilonhat}{\widehat{\epsilon}}
\newcommand{\epsilonbar}{{\overline{\epsilon}}}
\newcommand{\kappabar}{\overline{\kappa}}
\newcommand{\gammabar}{\overline{\gamma}}
\renewcommand{\equiv}{:=}

\title{Convergence of Proximal Splitting Algorithms in CAT($\kappa$) Spaces and Beyond}
\author{
{Florian Lauster}
 \thanks{Institute for Numerical and Applied Mathematics, 
University of G\"ottingen. FL was supported by 
the Deutsche Forschungsgemeinschaft (DFG, German Research Foundation) – 
Project-ID LU 1702/1-1.  
\texttt{f.lauster@math.uni-goettingen.de}}, 
\and
{D. Russell Luke} 
\thanks{Institute for Numerical and Applied Mathematics,
    University of Goettingen,
    37083 Goettingen, Germany. DRL was supported in part by 
    the Deutsche Forschungsgemeinschaft (DFG, German Research Foundation) – Project-ID 432680300 – SFB 1456.
	\texttt{r.luke@math.uni-goettingen.de}}
}

\date{\today}

\begin{document}
 \maketitle

 \begin{abstract}
In the setting of CAT($\kappa$) spaces, common fixed point iterations built from prox mappings   
(e.g. prox-prox, Krasnoselsky-Mann relaxations, nonlinear projected-gradients) 
converge locally linearly under the assumption of linear metric subregularity.  Linear metric
subregularity is in any case necessary for linearly convergent fixed point sequences, 
so the result is tight.  To show this, we develop a theory of fixed point mappings that violate 
the usual assumptions of nonexpansiveness and firm nonexpansiveness in $p$-uniformly 
convex spaces.  
\end{abstract}

{\small \noindent {\bfseries 2010 Mathematics Subject Classification:}
  Primary 
  47H09, 
 47H10, 
 53C22,  
 53C21   

 Secondary 
 53C23,  
 53C20,  
 49M27
  }


\section{Fundamentals of Nonlinear Spaces}\label{s:1}
  
Following \cite{RuiLopNic15}
we focus on {\em $p$-uniformly convex spaces with parameter $c$} \cite{NaoSil11}:  
for $p\in (1,\infty)$,  a metric space $(G, d)$ is $p$-uniformly convex with constant 
$c>0$  whenever it is a geodesic space, and 
\begin{equation}\label{e:p-ucvx}
(\forall t\in [0,1])(\forall x,y,z\in G) \quad
d(z, (1-t)x\oplus ty)^p\leq (1-t)d(z,x)^p+td(z,y)^p - \tfrac{c}{2}t(1-t)d(x,y)^p.
\end{equation}
Examples of $p$-uniformly convex spaces are $L^p$ spaces, CAT(0) spaces ($p=c=2$), 
Hadamard spaces (complete CAT$(0)$ spaces), Hilbert spaces (linear Hadamard spaces).  
Of particular interest are CAT($\kappa$) spaces since these serve as the model space for 
applications on manifolds with curvature bounded above.  
\begin{lemma}[\cite{Ohta}, Proposition 3.1]\label{t:CATkappa-pucvx}
A CAT$(\kappa)$ space is locally $2$-uniformly convex with parameter $c\nearrow 2$ as the 
diameter of the local neighborhood vanishes.  In particular, for any CAT$(\kappa)$ space
$(G, d)$ and any point $\xbar\in G$, for all $\delta\in (0, \pi/(4\sqrt{\kappa}))$ the subspace 
$(\Ball_\delta(\xbar), \left.d\right|_{\Ball_\delta(\xbar)})$ is a $2$-uniformly 
convex space with constant $c_\delta = 4\delta\sqrt{\kappa}\tan\paren{\pi/2 - 2\delta\sqrt{\kappa}}$.  
\end{lemma}
Note the asymptotic behavior of the constants:
as $\delta\searrow 0$,  the constant $c\nearrow 2$. 

\begin{mydef}
Let $(G,d)$ be a geodesic space and $\gamma$ and $\eta$ be two geodesics through $p$. 
Then $\gamma$ is said to be perpendicular to $\eta$ at point $p$ denoted by $\gamma \perp_p \eta$ if
\begin{align*}
d(x,p)\leq d(x,y) \quad \forall x\in \gamma, y \in \eta
\end{align*}
A space is said to be symmetric perpendicular if for all geodesics $\gamma$ and $\eta$ with common point $p$ we have
\begin{align*}
\gamma \perp_p \eta \Leftrightarrow \eta \perp_p \gamma.
\end{align*}
\end{mydef}

\begin{remark} \label{r:remsymmetric}
Any $CAT(\kappa)$ space $(G,d)$ with $diam(G) < \frac{\pi}{2\sqrt{\kappa}} $ is symmetric perpendicular 
\cite[Theorem 2.11]{Kuwae2014}.
\end{remark}

In a complete $p$-uniformly convex space the 
$p$-proximal mapping of a proper and lower semicontinuous function $f$ is defined by
\begin{align}\label{e:prox^p}
\prox^p_{f, \lambda} (x) \coloneqq \argmin_{y \in G} f(y)+ \frac{1}{p \lambda^{p-1}} d(y,x)^p \quad(\lambda>0). 
\end{align}
The main dividend of this work is the following.
\begin{thm}[convergence of proximal algorithms in CAT($\kappa$) spaces]
	\label{t:convergence ppa-proto}
	Let $(G, d)$ be a $CAT(\kappa)$ space with $\kappa >0$ and for $j=1,2,\dots,N$, let   
	$f_j$ and $g\colon G \rightarrow \Rbb\cap \{+\infty\}$ be  
proper, convex and lower semicontinuous. 
Let $\mymap{T}{D}{D}$ where $D\subset G$ denote one of the following:
\begin{enumerate}[(i)]
	\item $T\equiv\prox_{f_N,\lambda_N}\circ \prox_{f_{N-1},\lambda_{N-1}}\circ\dots\circ\prox_{f_1,\lambda_1}$;
	\item $T\equiv \beta\prox_{g,\lambda}\oplus (1-\beta)\Id$;
	\item $T\equiv \prox_{f_1,\lambda_1}\circ\paren{\beta\prox_{g,\lambda_2}\oplus (1-\beta)\Id}$;
	\item $T\equiv P_C\circ\paren{\beta\prox_{g,\lambda_1}\oplus (1-\beta)\Id}$,
\end{enumerate}
where in the last case $P_C$ is the metric projector onto the closed convex set $C$. 
If  $T$ satisfies $\Fix T\neq\emptyset$ and 
\[
 d(x,\Fix T\cap D)\leq \mu d(x,Tx),\hspace{0.2cm}\forall x\in D\subset G, 
\]
with constant  $\mu$, then the fixed point sequence  initialized from any starting point close 
enough to $\Fix T$ is at least linearly convergent 
to a point in $\Fix T$.
\end{thm}
A more precise statement of this theorem, with proof,  is Theorem \ref{t:convergence ppa}. 
The intervening sections prove the fundamental building blocks.  

\section{Almost $\alpha$-firmly Nonexpansive Mappings}\label{s:Metric}
The regularity of a mapping $\mymap{T}{G}{G}$ is characterized by the behavior of the images of 
pairs of points under $T$.  A key tool is what has been called the {\em transport discrepancy} in \cite{BLL}: 
\begin{equation}\label{eq:psi}
\quad \psi^{(p,c)}_T(x,y) \coloneqq 
\tfrac{c}{2}\paren{d(Tx, x)^p+d(Ty, y)^p + d(Tx, Ty)^p + d(x, y)^p  - d(Tx, y)^p   - d(x,Ty)^p }.
\end{equation}

\begin{mydef}\label{d:pafne}
 Let $(G, d)$ be a $p$-uniformly convex metric space with constant $c$.  
 \begin{enumerate}[(i)]
  \item The mapping $T:G\to G$ 
 is pointwise almost nonexpansive at $y\in D\subset G$ on $D$ with violation $\epsilon>0$
 whenever 
\begin{equation}
\label{eq:pane}
\exists \alpha\in(0,1), \epsilon \geq 0:\quad  d(Tx,Ty)^p\leq (1+\epsilon)d(x,y)^p\quad 
\forall x\in D.
\end{equation}
The smallest $\epsilon$ 
for which \eqref{eq:pafne} holds is called the {\em violation}. 
If \eqref{eq:pane} holds with $\epsilon=0$, then 
$T$ is pointwise nonexpansive at $y\in D\subset G$ on $D$. 
If \eqref{eq:pane} holds at all $y\in D$ then $T$ is said to be (almost) nonexpansive on $D$.  
If $D=G$ the mapping 
$T$ is simply said to be (almost) nonexpansive.  If $D\supset \Fix T\neq\emptyset$ and 
\eqref{eq:pane} holds at all $y\in \Fix T$ with the same violation, then $T$ is 
said to be almost quasi nonexpansive.  
\item  The mapping $T$ is said to be 
{\em  pointwise asymptotically nonexpansive at $y$} whenever 
\begin{equation}
\label{eq:asymp-pane}
\forall \epsilon>0, \exists D_\epsilon(y)\subset G:
\quad  d(Tx,Ty)^p\leq (1+\epsilon)d(x,y)^p\quad 
\forall x\in D_\epsilon(y),
\end{equation}
where $D_\epsilon(y)$ is a neighborhood of $y$ in $D$.
\item $\mymap{T}{G}{G}$ is said to be {\em quasi strictly nonexpansive} whenever
\begin{equation}\label{eq:qsne}
d(Tx,\xbar)< d(x,\bar{x})\quad \forall x \in G, \forall \xbar\in \Fix T.   
\end{equation}
\item  The operator $T:G\to G$ 
is {\em pointwise almost $\alpha$-firmly nonexpansive at $y\in D\subset G$ on $D$ 
	with violation $\epsilon>0$}
whenever
\begin{equation}
\label{eq:pafne}
\exists \alpha\in(0,1), \epsilon \geq 0:
\quad  d(Tx,Ty)^p\leq (1+\epsilon)d(x,y)^p-\tfrac{1-\alpha}{\alpha}\psi^{(p,c)}_T(x,y)\quad 
\end{equation}
If \eqref{eq:pafne} holds with $\epsilon=0$, then 
$T$ is pointwise $\alpha$-firmly nonexpansive at $y\in D\subset G$ on $D$. 
If \eqref{eq:pafne} holds at all $y\in D$ with the same constant $\alpha$, 
then $T$ is said to be (almost) $\alpha$-firmly nonexpansive on $D$.  If $D=G$ the mapping 
$T$ is simply said to be (almost) $\alpha$-firmly nonexpansive.  If $D\supset \Fix T\neq\emptyset$ and 
\eqref{eq:pafne} holds at all $y\in \Fix T$ with the same constant $\alpha$ then $T$ is 
said to be almost quasi $\alpha$-firmly nonexpansive.  
\item  The mapping $T$ is said to be {\em pointwise asymptotically  $\alpha$-firmly 
	nonexpansive at $y$ with constant $\alpha<1$} whenever 
\begin{equation}
\label{eq:asymp-pafne}
\forall \epsilon>0, ~\exists D_\epsilon(y)\subset G:\quad
d(Tx,Ty)^p\leq (1+\epsilon)d(x,y)^p-\tfrac{1-\alpha}{\alpha}\psi^{(p,c)}_T(x,y)\quad 
\forall x\in D_\epsilon(y).
\end{equation}
  where $D_\epsilon(y)$ is a neighborhood of $y$ in $D$.
 \end{enumerate}
\end{mydef}

\begin{proposition}[characterizations]
\label{t:properties pafne}
Let  $(G, d)$ be a p-uniformly convex space with constant $c>0$ and   
let $T:D\to G$ for $D\subset G$.
\begin{enumerate}[(i)]
\item\label{t:properties pafne ii}
\begin{equation}\label{e:psi-Fix T}
\psi_T^{(p,c)}(x, y)=\tfrac{c}{2}d(Tx, x)^p \quad
\mbox{whenever } y\in\Fix T.
\end{equation}
For fixed $y\in\Fix T$ the function $\psi_T^{(p,c)}(x,y)\geq 0$ for all $x \in D$ and  $\psi_T^{(p,c)}(x,y)= 0$ 
only when  $x\in\Fix T$.
\item\label{t:properties pafne iii}  Let $y\in \Fix T$.  $T$ is  pointwise almost $\alpha$-firmly nonexpansive at $y$ 
on $D$ with violation $\epsilon>0$ if and only if  
\begin{equation}\label{e:P1}
\exists \alpha\in[0,1):\quad 	d(Tx,y)^p\leq (1+\epsilon)d(x,y)^p - \tfrac{1-\alpha}{\alpha}\tfrac{c}{2}d(Tx,x)^p\quad
	\forall x\in D.
\end{equation}
In particular, $T$ is almost quasi $\alpha$-firmly nonexpansive on $D$ whenever $T$ possesses fixed points and 
\eqref{e:P1} holds at all $y\in \Fix T$ with the same constant $\alpha\in [0,1)$ and violation $\epsilon$. 
\item\label{t:properties pafne iv}  If $T$ is pointwise almost $\alpha$-firmly nonexpansive at $y\in\Fix T$ on $D$ with 
constant $\underline\alpha\in[0,1)$ and violation $\epsilon$, then it is pointwise almost $\alpha$-firmly nonexpansive 
at $y$ with the same violation on $D$ for all 
$\alpha\in[\underline\alpha,1]$.   In particular, if $T$ is pointwise almost $\alpha$-firmly nonexpansive at $y\in\Fix T$ on D, 
then it is pointwise almost nonexpansive at $y$ with violation $\epsilon$ on D.  
\end{enumerate}
\end{proposition}
\begin{proof}
This is a slight extension of \cite[Proposition 4]{BLL}, which was for pointwise $\alpha$-firmly
nonexpansive mappings.  The proof for pointwise almost $\alpha$-firmly 
nonexpansive mappings is the same.  
\end{proof}

\subsection{Composition of Operators}
Before continuing with pointwise almost $\alpha$-firmly nonexpansive mappings, we make a 
brief but important observation about fixed points of compositions of quasi strictly nonexpansive 
mappings \eqref{eq:qsne}.  
\begin{lemma}\label{t:intersections}
Let $T_1$ and $T_2$ be quasi strictly nonexpansive on $(G,d)$ with 
$\Fix T_1 \cap \Fix T_2 \neq \emptyset$. Then $\Fix T_1 \circ T_2 = \Fix T_1 \cap \Fix T_2$.
\end{lemma}
\begin{proof}
	The inclusion $\Fix T_1 \cap \Fix T_2  \subset \Fix \paren{T_1 \circ T_2}$ is clear. 
	Assume there exists a $y \in \Fix ( T_1 \circ T_2) \setminus (\Fix T_1 \cap \Fix T_2)$ and choose $x \in \Fix T_1 \cap \Fix T_2$. 
	Then 
$$d(y,x)\leq d(T_1 \circ T_2 y,x) \leq d(T_2y,x) \leq d(y,x)$$
with either $d(T_2y,x) < d(y,x)$ or $d(T_1 \circ T_2 y,x) < d(T_2 y,x)$ as $y  \notin \Fix T_1 \cap \Fix T_2$. 
This is a contradiction,  so $\Fix T_1 \circ T_2 = \Fix T_1 \cap \Fix T_2$.
\end{proof}

\begin{remark}
The sufficiency of strict quasi-nonexpansivity for the analogous identity for convex combinations of 
mappings in a Hadamard space was recognized in \cite[Remark 7.11]{Berdellima_PhD}. 
\end{remark}

\begin{lemma}\label{t:afne of compositions}
 Let $(G, d)$ be a $p$-uniformly convex space with constant $c>0$ and let $D\subset G$.  
 Let  ${T_0}:D\to G$ be pointwise almost $\alpha$-firmly nonexpansive at $y$ on 
 $D$ with constant $\alpha_0$ and violation $\epsilon_0$
 and let ${T_1}:{T_0}(D)\to G$ be pointwise almost $\alpha$-firmly nonexpansive at ${T_0} y$ on 
 ${T_0}(D)$ with constant 
 $\alpha_1$ and violation $\epsilon_1$.  
 Then the composition $\Tbar\equiv {T_1}\circ {T_0}$ is pointwise almost $\alpha$-firmly nonexpansive 
 at $y$ with constant $\alphabar \in (0,1)$ and violation 
 $\epsilonbar=\epsilon_0+\epsilon_1 + \epsilon_0\epsilon_1$ on $D$ whenever 
 \begin{equation}
\label{eq:afne composition}
\frac{1-\alphabar}{\alphabar}\psi^{(p,c)}_{\Tbar}(x,y) \leq (1+\epsilon_1)
\frac{1-\alpha_0}{\alpha_0}\psi^{(p,c)}_{{T_0}}(x,y)
 + \frac{1-\alpha_1}{\alpha_1}\psi^{(p,c)}_{{T_1}}({T_0}x,{T_0}y)
\quad \forall x\in D.
\end{equation}
\end{lemma}
\begin{proof}
The proof is a slight extension of the same result for compositions of $\alpha$-firmly nonexpansive mappings 
in \cite[Lemma 10]{BLL}.  Since ${T_1}$ is pointwise $\alpha$-firmly nonexpansive at ${T_0} y$ with violation $\epsilon_1$ and
constant $\alpha_1$ on ${T_0}(D)$ we have 
\[ 
d(\Tbar x,\Tbar y)^p\leq  (1+\epsilon_1) d({T_0} x,{T_0} y)^p-\frac{1-\alpha_1}{\alpha_1}\psi^{(p,c)}_{{T_1}}({T_0}x,{T_0}y)
 \quad \forall x \in D
\]
where $\psi^{(p,c)}_{{T_1}}$ is defined by \eqref{eq:psi}. Since ${T_0}$ is $\alpha$-firmly 
nonexpansive at $y$ with constant $\alpha_0$ with violation $\epsilon_0$ on $D$ we have

\begin{align*}
d(\Tbar x,\Tbar y)^p\leq  (1+\epsilon_0)(1+\epsilon_1) d(x,y)^p&
-(1+\epsilon_1) \frac{1-\alpha_0}{\alpha_0}\psi^{(p,c)}_{{T_0}}(x,y)
-\frac{1-\alpha_1}{\alpha_1}\psi^{(p,c)}_{{T_1}}({T_0}x,{T_0}y),
\end{align*}
for all $x \in D$.
Whenever \eqref{eq:afne composition} holds, we conclude that  
\begin{eqnarray*}
\exists~ \alphabar\in(0,1):\quad 
d(\Tbar x,\Tbar y)^p\leq
(1+\epsilonbar)d(x,y)^p-
\frac{1-\alphabar}{\alphabar}\psi^{(p,c)}_{\Tbar }(x,y)
\quad\forall x\in D,
\end{eqnarray*}
where $\epsilonbar = \epsilon_0+\epsilon_1+\epsilon_0\epsilon_1$.
\end{proof}

 \begin{proposition}[compositions of pointwise almost $\alpha$-firmly nonexpansive mappings]
 \label{t:compositionthm}
 Let $(G, d)$ be a $p$-uniformly convex space with constant $c>0$ and let $D\subset G$.  
 Let  $T_0:D\to G$ be pointwise almost $\alpha$-firmly nonexpansive at $y$ on 
 $D$ with constant $\alpha_{0}$ and violation $\epsilon_{0}$
 and let $T_1:T_0(D)\to G$ be pointwise almost $\alpha$-firmly nonexpansive at $y$ on 
 $T_0(D)$ with constant 
 $\alpha_{1}$ and violation $\epsilon_{1}$. Let $y \in \Fix T_0 \cap \Fix {T_1}$. 
 Then the composite operator $\Tbar={T_1} \circ T_0$ is pointwise almost $\alpha$-firmly nonexpansive at 
 $y$ on $D$ with violation  $\epsilonbar=\epsilon_{0}+\epsilon_{1} + \epsilon_{0}\epsilon_{1}$ and constant 
 \begin{equation}\label{eq:alphabar}
	 \alphabar=\tfrac{\alpha_{0}+\alpha_{1}-2\alpha_{0}\alpha_{1}}%
	 {\frac{c}{2}\paren{1-\alpha_{0}-\alpha_{1}+\alpha_{0}\alpha_{1}}+
		 \alpha_{0}+\alpha_{1}-2\alpha_{0}\alpha_{1}}.
 \end{equation}
\end{proposition}
\begin{proof}
	This is a minor extension of \cite[Theorem 11]{BLL}. 	
	By Lemma \ref{t:afne of compositions}, it suffices to show \eqref{eq:afne composition}
	at all points  $y\in \Fix {T_1}\cap \Fix {T_0}$.  First, note that   
	$\Fix \Tbar \supset \Fix {T_1}\cap \Fix {T_0}$, 
	so by \eqref{e:psi-Fix T} we have  
	$\psi^{(p.c)}_{T_0}(x,y)=\tfrac{c}{2}d(x,{T_0}x)^p$, $ \psi^{(p.c)}_{T_1}({T_0}x,{T_0}y)=
	\tfrac{c}{2}d({T_0}x,\Tbar x)^p$, and 
 $\psi^{(p.c)}_{\Tbar }(x,y)=\tfrac{c}{2}d(x,\Tbar x)^p$ whenever $y\in \Fix {T_1}\cap \Fix {T_0}$. 
Inequality \eqref{eq:afne composition} in this case simplifies to 
 \begin{equation}\label{eq:afne composition Fix TS}
\exists \kappabar>0:\quad	 
\kappa_{0} (1+\epsilon_{1})d(x, {T_0}x)^p+\kappa_{1} d({T_0}x, \Tbar x)^p
\geq \kappabar d(x, \Tbar x)^p  
\qquad \forall x\in D,
 \end{equation}
 where $\kappa_0\equiv\frac{1-\alpha_0}{\alpha_0}$, 
 $\kappa_{T_1}\equiv\frac{1-\alpha_1}{\alpha_1}$ and 
 $\kappabar\equiv\frac{1-\alphabar}{\alphabar}$ with $\alphabar\in(0,1)$.
By \eqref{e:p-ucvx}, we have 
\begin{eqnarray}
\tfrac{c}{2}t(1-t)d(x,\Tbar x)^p&\leq& \tfrac{c}{2}t(1-t)d(x,\Tbar x)^p+ 
d({T_0}x, (1-t)x\oplus t\Tbar x)^p\nonumber\\
&\leq& (1-t)d({T_0}x,x)^p+
td({T_0}x,\Tbar x)^p
\quad\forall x\in G, \forall t\in(0,1).
\label{e:p-ucvx2}
\end{eqnarray}
Letting $t=\tfrac{\kappa_1}{\kappa_0+\kappa_1}$ yields
$(1-t)=\tfrac{\kappa_0}{\kappa_0+\kappa_1}$, so that 
\eqref{e:p-ucvx2} becomes
\begin{eqnarray}
	\tfrac{c}{2}\tfrac{\kappa_0\kappa_1}{\kappa_0+\kappa_1}d(x,\Tbar x)^p
&\leq& (1+\epsilon_1) \kappa_0 d({T_0}x,x)^p+
\kappa_1 d({T_0}x,\Tbar x)^p
\quad\forall x\in G.
\label{e:p-ucvx3}
\end{eqnarray}
It follows that  \eqref{eq:afne composition Fix TS}  holds for any  
$\kappabar\in \left(0,\tfrac{c\kappa_0\kappa_1}{2(\kappa_0+\kappa_1)}\right]$.  
We conclude that the composition $\Tbar $ is quasi $\alpha$-firmly nonexpansive with constant 
\[
\alphabar=\frac{\kappa_0+\kappa_1}{\frac{c}{2}\kappa_0\kappa_1+\kappa_0+\kappa_1}. 
\]
A short calculation shows that this is the same as \eqref{eq:alphabar}, which completes the proof.
\end{proof}

\subsection{Averages of Operators}
Let $\mathcal{B}(G)$ be the borel algebra on $(G,d)$, $\mathcal{P}$ the family of probability measures on $(G,\mathcal{B}(G))$ and  $\mathcal{P}^{p}(G)$ the family of probability measures on $G$ such that the $p$-th moment exists, i.e.:
\begin{align*}
\mathcal{P}^{p}(G) \coloneqq \{ \nu \in  \mathcal{P}(G) \mid  \int_y d(x,y)^{p} \nu(dx) < \infty \quad \forall x \in G\}.
\end{align*}
For $\mu \in \mathcal{P}^p(G)$ the minimizer of 
\begin{align*}
x\mapsto \int_G d(x,y)^p \nu (dy)
\end{align*}
is called $p$-barycenter of $\nu$ and denoted by $b_p(\nu)$ if it exists. The $p$-barycenter of $\nu$ always exsists if $(G,d)$ is a proper geodesic space and $\mu \in \mathcal{P}^{p}(G)$ \cite[Proposition 3.3]{Kuwae2014}.

Let $T_i \colon G \rightarrow G, i\in I$ be a collection of mappings where $I$ is an index space. Assume that $(I,\mathcal{I})$ is a measurable space and $(x,i) \mapsto T_ix$ is measurable. Let $\eta$ be a probability measure on $I$ and define $b_p(T_i,\eta)\colon G \rightarrow G$ by
\begin{align*}
b_p(T_i,\eta)(x) = \argmin_{y \in G} y \mapsto \int_I d(T_i x,y)^p \eta(di).
\end{align*}

We use the notation $T_ix_*\eta$ for the push forward of $\eta$ with respect to the mapping $i \mapsto T_i x$ for fixed $x$, i.e.
\begin{align}
T_ix_*\eta (A) = \eta (\{i \mid T_ix \in A\}) \label{eq:nu}
\end{align}  for $A \in \mathcal{B}(G)$. Then by definition $b_p(T_i,\eta)(x) = b_p(T_i x _*\eta)$.

\begin{thm}
Let $G$ be a proper, symmetric perpendicular, $p$-uniformly convex space with constant $c>0$ and let $T_i$, $i \in I$ be a family of almost quasi $\alpha$-firmly 
nonexpansive operators with violation $\epsilon_i$ and constant $\alpha_i$ respectively on $G$. Let $\eta$ be a probability measure on $I$ such that $T_ix_*\eta \in  \mathcal{P}^{p}(G) $ for all $x\in G$. Then $\mathscr{T}=b_p(T_i,\eta)$ is a pointwise almost $\alpha$-firm operator at any $y \in \cap_{i\in I} \Fix T_i$ on $G$ with 
constant $\alphabar=\sup_{i\in I} \alpha_i$ and violation $\epsilonbar= \sup_{i\in I} \epsilon_i$.
\end{thm}

\begin{proof}
Let $x \in G$ be arbitrary, $y \in \cap_{i\in I} \Fix T_i \subset \Fix \mathscr{T} $, $\nu = T_i x _*\eta$ as defined in (\ref{eq:nu}) and 
$(d(\cdot,y)^p) _* \nu$ the push forward of $\nu$ . Then 
\begin{align*}
d(\mathscr{T} x,\mathscr{T} y)^p = d(b_p(T_ix_* \eta),y)^p = d(b_p(\nu),y)^p
\leq b_p((d(\cdot,y)^p) _* \nu)
\end{align*}
by Jensens inequality \cite[Theorem 4.1]{Kuwae2014} since $d(\cdot,y)^p$ is a convex function and the 
fact that $\Rbb$ is a $p$-uniformly convex space with constant $c$.
Now
\begin{align*}
b_p((d(\cdot,y)^p) _* \nu) 
&= \argmin_{t \in \Rbb} \int_z \vert t- d(z,y)^p \vert^p \nu(dz) \\
&= \argmin_{t \in \Rbb} \int_i \vert t- d(T_ix,y)^p \vert^p \eta(di) \\
&\leq \argmin_{t \in \Rbb} \int_i \vert t- [(1+\epsilon_i) d(x,y)^p- \frac{1-\alpha_i}{\alpha_i} \frac{c}{2} d(T_i x,x)^p] \vert^p \eta(di) \\\
&\leq (1+\epsilonbar) d(x,y)^p - \argmin_{t \in \Rbb} \int_i \vert t - \frac{1-\alphabar}{\alphabar} \frac{c}{2} d(T_ix,x)^p \vert^p \eta(di) \\
&= (1+\epsilonbar) d(x,y)^p -  \frac{1-\alphabar}{\alphabar} \frac{c}{2} \argmin_{t \in \Rbb} \int_z \vert t - d(z,x)^p \vert^p \nu(dz) \\
&=(1+\epsilonbar) d(x,y)^p -  \frac{1-\alphabar}{\alphabar} \frac{c}{2}  b_p ((d(\cdot,y)^p) _* \nu).
\end{align*}
And again Jensens inequality completes the proof
\begin{align*}
(1+\epsilonbar) d(x,y)^p - b_p (\frac{1-\alphabar}{\alphabar} \frac{c}{2} d(x,T_ix)_* \eta) \\
\leq (1+\epsilonbar) d(x,y)^p - \frac{1-\alphabar}{\alphabar} \frac{c}{2} d(x, b_p (T_ix_* \eta))^p.
\end{align*}
\end{proof}
For a finite index set $I$, without loss of generality $I=\{1,\ldots,n\}$ and a probability measure $\eta$ on $I$ we can define $\omega_i \coloneqq \eta(i)$ for all $i$. Then 
\begin{align*}
b_p(T_i,\eta)(x)=\argmin_{z\in G} \sum_{i=1}^n \omega_i d(z,T_ix)^p.
\end{align*}
In case of a Hilbert space $G$ and $p=c=2$ this reduces further to $b_p(T_i,\eta)(x) = \sum_{i=1}^n \omega_i T_i x$.

If the support of the measure $\nu$ consits of two disrecte points $x_1$ and $x_2$, i.e. for $\omega \in [0,1]$ $\nu = \omega \delta_{x_1}+(1-\omega)\delta_{x_2}$, then $b_p(\nu)$ can be calculated explicitly. It is obvious that $b_p(\nu)$ has to lie on the geodesic connecting $x_1$ and $x_2$. Hence $ b_p(\nu) = \tbar x_1 \oplus (1-\tbar) x_2$ for some $\tbar \in [0,1]$. Minimizing the function $t\mapsto \omega d( t x_1 \oplus (1-t) x_2 , x_1)^p+ (1-\omega) d( t x_1 \oplus (1-t) x_2 , x_1)^p$ leads to $\tbar = \frac{1}{\sqrt[p-1]{\frac{1-\omega}{\omega}}+1}$. If $I=\{1,2\}$, $T_1=T$ and $T_2=Id$, $\eta=\omega \delta_1(\cdot)+ (1-\omega) \delta_2(\cdot)$ for $\omega \in [0,1]$ then $T_\beta=b_p(T_i,\eta)$ is Krasnoselsky-Mann relaxation of $T$  
$$T_\beta=\beta T \oplus (1-\beta) Id \quad \text{with } \beta=\frac{1}{\sqrt[p-1]{\frac{1-\omega}{\omega}}+1}.$$

The next result shows that the convex combination of an almost nonexpansive mapping with the identity mapping 
can be made arbitrarily  close to $\alpha$-firmly nonexpansive (no violation) by choosing the 
averaging constant small enough  -- this can be interpreted as choosing an appropriately small step size.    
\begin{proposition}[Krasnoselsky-Mann relaxations]\label{t:KM}
Let $(G,d)$ be a $p$-uniformly convex space and $T\colon G \rightarrow G$ 
be pointwise almost nonexpansive at all $y \in \Fix T$ with violation $\epsilon$. Then 
$T_\beta \coloneqq \beta T \oplus (1-\beta) Id$ is pointwise almost $\alpha$-firmly 
nonexpansive at all $y \in \Fix T$ with constant 
\[
\alpha_\beta = \frac{\alpha\beta^{p-1}}{\alpha\beta^{p-1} -\alpha\beta+ 1}
\]
and violation $\epsilon_\beta\equiv \epsilon\beta$.
\end{proposition}
\begin{proof}
Clearly $\Fix T = \Fix T_\beta$ and $d(x,T_\beta x)^p=\beta^p d(x,Tx)^p$. 
Let $y \in Fix T_\beta$ then 
\begin{align*}
d(y,T_\beta x)^p &=d(y, \beta Tx \oplus (1-\beta)x)^p \\
&\leq \beta d(y,Tx)^p+(1-\beta) d(y,x)^p- \frac{c}{2} \beta (1-\beta) d(x,Tx)^p \\
& \leq (1+\epsilon\beta) d(x,y)^p - \frac{c}{2}\beta^{1-p}\paren{\tfrac{1-\alpha}{\alpha}+1-\beta} d(x,T_\beta x)^p.
\end{align*}
Setting 
\[
\tfrac{1-\alpha_\beta}{\alpha_\beta} = \beta^{1-p}\paren{\tfrac{1-\alpha}{\alpha}+1-\beta}
\]
and solving for $\alpha_\beta$ yields the result.  
\end{proof}

\subsection{Metric Subregularity}
 Recall that $\mu:[0,\infty) \to [0,\infty)$ is a \textit{gauge function} if 
$\mu$ is continuous, strictly increasing 
with $\mu(0)=0$, and $\lim_{t\to \infty}\mu(t)=\infty$. 
\begin{mydef}[metric regularity on a set]\label{d:msr}
$~$ Let $(G_1, d_1)$ and $(G_2, d_2)$ be metric spaces and let  $\mmap{\Phi}{G_1}{G_2}$, 
$U\subset G_1$, $V\subset G_2$.
The mapping $\Phi$ is called \emph{metrically regular with gauge $\mu$ on $U\times V$ relative to 
$\Lambda\subset G_1$} if
\begin{equation}\label{e:metricregularity}
\forall y\in V, \forall x\in U\cap \Lambda,\quad 
\dist\paren{x, \Phi^{-1}(y)\cap \Lambda}\leq \mu\paren{\dist\paren{y, \Phi(x)}}.
\end{equation}
When the set $V$ consists of a single point, $V=\{\ybar\}$, then $\Phi$ is said to be 
\emph{metrically subregular for $\ybar$ on $U$ with gauge $\mu$ relative to $\Lambda\subset G_1$}.

When $\mu$ is a linear function (that is, $\mu(t)=\kappa t,\, \forall t\in [0,\infty)$), this 
special case is distinguished as {\em linear metric (sub)regularity with constant $\kappa$}. 
When $\Lambda=G_1$, the quantifier ``relative to'' is dropped.  
When $\mu$ is linear, the infimum of all constants  
$\kappa$ for which \eqref{e:metricregularity} holds is called the {\em modulus} of metric regularity.
\end{mydef}

\noindent The next statement is obvious from the definition. 
\begin{proposition}\label{t:persistence of msr}
Let $(G_1, d_1)$ and $(G_2, d_2)$ be metric spaces and let  $\mmap{\Phi}{G_1}{G_2}$, 
$U\subset G_1$, $V\subset G_2$. If  $\Phi$ is metrically subregular with gauge 
$\mu$ at $y$ on $U$ relative to 
$\Lambda\subset G_1$, then $\Phi$ is metrically subregular with the same gauge 
$\mu$ at $y$ on all subsets $U'\subset U$ relative to 
$\Lambda\subset G_1$.
\end{proposition}

\section{Quantitative Convergence}
To obtain convergence of fixed point iterations under the assumption of metric subregularity, 
the gauge of metric subregularity $\mu$  is constructed 
implicitly from another 
nonnegative function $\mymap{\theta}{[0,\infty)}{[0,\infty)}$ satisfying
\begin{eqnarray}\label{eq:theta}
(i)~ \theta(0)=0; \quad (ii)~ 0<\theta(t)<t ~\forall t>0; 
\quad (iii)~\sum_{j=1}^\infty\theta^{(j)}(t)<\infty~\forall t\geq 0.
\end{eqnarray}
For a $p$-uniformly convex space the operative gauge satisfies 
\begin{equation}\label{eq:gauge}
 \mu\paren{\paren{\frac{(1+\epsilon)t^p-\paren{\theta(t)}^p}{\tau}}^{1/p}}=
 t\quad\iff\quad
 \theta(t) = \paren{(1+\epsilon)t^p - \tau\paren{\mu^{-1}(t)}^p}^{1/p}
\end{equation}
for $\tau>0$ fixed and $\theta$ satisfying \eqref{eq:theta}.  

In the case of linear metric subregularity on a CAT($\kappa$) space
this becomes 
\[
t\mapsto \mu t\quad\iff\quad  
\theta(t)=\paren{(1+\epsilon)-\frac{\tau}{\mu^2}}^{1/2}t\quad 
(\mu\geq \sqrt{\tfrac{\tau}{(1+\epsilon)}}).
\]  
If 
\eqref{e:metricregularity} is satisfied for some $\mu'>0$, then 
the condition $\mu\geq \sqrt{\tfrac{\tau}{(1+\epsilon)}}$ is satisfied
for all $\mu\geq \mu'$ large enough. The conditions in \eqref{eq:theta} in this 
case simplify to $\theta(t)=\gamma t$ where 
\begin{equation}\label{eq:theta linear}
	0< \gamma \equiv \sqrt{1+\epsilon-\frac{\tau}{\mu^2}}<1\quad\iff\quad 
\sqrt{\tfrac{\tau}{(1+\epsilon)}}\leq  \mu \leq \sqrt{\tfrac{\tau}{\epsilon}}.
\end{equation}

The next definition 
characterizes the quantitative convergence of sequences in terms of gauge functions. 
 \begin{mydef}[gauge monotonicity \cite{LukTebTha18}]\label{d:mu_Mon}
Let $(G,d)$ be a metric space, let $(x_k)_{k\in\Nbb}$ be a sequence on $G$, 
let $D\subset G$ be nonempty and let the continuous mapping
$\mymap{\mu}{\Rbb_+}{\Rbb_+}$ satisfy $\mu(0)=0$ and
\begin{align*}
&\mu(t_1)<\mu(t_2)\leq t_2\; \mbox{ whenever }\; 0\le t_1<t_2.
\end{align*}
\begin{enumerate}[(i)]
 \item $(x_k)_{k\in \Nbb}$ is said to be
 \emph{gauge monotone with respect to $D$ with rate $\mu$}  whenever
 \begin{equation}\label{e:mu-uniform mon}
 d(x_{k+1}, D)\leq \mu\paren{d(x_k, D)}\quad  \forall k\in \Nbb .
 \end{equation}
 \item $(x_k)_{k\in \Nbb}$ is said to be
 \emph{linearly monotone with respect to $D$} with rate $c$ if \eqref{e:mu-uniform mon} is
 satisfied for $\mu(t)=c\cdot t$ for all $t\in \Rbb_+$ and some constant $c\in [0,1]$.
 \end{enumerate}
A  sequence $(x_k)_{k\in \Nbb}$ is said to converge 
{\em gauge monotonically} 
to some element  $x^*\in G$ with rate 
$s_k(t)\equiv \sum_{j=k}^\infty \mu^{(j)}(t)$ whenever 
it is gauge monotone with gauge $\mu$ satisfying 
$\sum_{j=1}^\infty\mu^{(j)}(t)<\infty~\forall t\geq 0$, 
and 
there exists a constant 
$a>0$  such that $d(x_k,x^*)\leq  a s_k(t)$ for all $k\in \mathbb{N}$. 
\end{mydef}
All Fej\'er monotone sequences are linearly monotone (with constant $c=1$)
but the converse does not hold (see Proposition 1  and Example 1 of \cite{LukTebTha18}).
Gauge-monotonic convergence for a linear gauge in the definition above is just 
$R$-linear convegence.  

Metric subregularity and pointwise (almost) nonexpansiveness are fundamentally connected through 
the surrogate mapping 
$\mymap{\Tcal_S}{G}{\mathbb{R}_+}\cup\{+\infty\}$ 
defined  by 
 \begin{equation}\label{eq:Tcal}
	 \Tcal_S(x)\equiv 
	 \left|\paren{\inf_{y\in S}\psi^{(p,c)}_T(x,y) }^{1/p}\right|
 \end{equation}
 where $\psi^{(p,c)}_T$ is defined by \eqref{eq:psi} and   $S\subset G$.  
 If $S=\emptyset$ then, by definition, $\Tcal_S(x)\equiv+\infty$ for all $x$.  
 Hence, $\Tcal_S$ is proper  when  $S$ is nonempty.   
 For our purposes, $S\subseteq\Fix T$, in which case by Proposition \ref{t:properties pafne}\eqref{t:properties pafne ii} 
 we have  
$ \psi^{(p,c)}_T(x,y)\geq 0$ for all $x\in D$ and all $y\in S$ and 
$ \psi^{(p,c)}_T(x,y)= 0$ only when both $x,y\in\Fix T$.  Hence 
$\Tcal_S$ is nonnegative, and takes the value $0$ only on $\Fix T$ and has the simple representation  
 \begin{equation}\label{eq:Tcal_Fix_T}
	 \Tcal_S(x)= \sqrt[p]{\tfrac{2}{c}}d(Tx,x)>0\quad (S\neq\emptyset).  
 \end{equation}

\begin{thm}[necessary and sufficient conditions for convergence rates]\label{t:msr convergence}
 Let $(G,d)$ be a $p$-uniformly convex space with constant $c$; let $D\subset G$, 
 let $\mymap{T}{D}{D}$,  
 and let $S\equiv \Fix T\cap D$ be nonempty. 
  Assume further that $T$ is pointwise almost $\alpha$-firmly nonexpansive at all 
 points $y\in S$ 
 with the same constant $\alphabar$ and violation $\epsilon$ on $D$.

 \begin{enumerate}[(a)]
\item\label{t:msr convergence, necessary} (necessity) Suppose that
all sequences $(x^k)_{k\in\mathbb{N}}$ defined by $x^{k+1}=Tx^k$ 
and initialized in $D$ are gauge monotone 
relative to $S$ 
with rate $\theta$ satisfying \eqref{eq:theta}, and 
$(\Id - \theta)^{-1}(\cdot)$ is continuous on $\Rbb_+$, strictly increasing, 
and $(\Id - \theta)^{-1}(0)=0$.  Then all sequences initialized on $D$ 
converge gauge monotonically to some $\xbar\in S$ with rate 
$O(s_k(t_0))$ where 
$s_k(t)\equiv 
\sum_{j=k}^\infty \theta^{(j)}(t)$ and $t_0\equiv d(x^0,\Fix T\cap D)$.
Moreover, $\Tcal_{S}$ defined by 
\eqref{eq:Tcal} is metrically subregular for $0$ relative to $D$ on $D$ 
with gauge $\mu(\cdot)=(\Id-\theta)^{-1}(\cdot)$. 
\item\label{t:msr convergence, sufficient} (sufficiency) If 
 $T$ satisfies 
 \begin{equation}
 \label{eq:estimate1}
 d(x,\Fix T\cap D)\leq \mu(d(x,Tx)),\hspace{0.2cm}\forall x\in D,
 \end{equation}
with gauge $\mu$ given by \eqref{eq:gauge} 
for  $\tau=(1-\alphabar)/\alphabar$ and $\epsilon>0$ the violation of pointwise $\alpha$ 
 firmness of $T$ on $D$, 
then for any $x^0\in D$, the sequence $(x^k)_{k\in\Nbb}$ defined by 
$x^{k+1}= T x^k$ satisfies 
\begin{equation}\label{eq:gauge convergence}
d\paren{x^{k+1},\Fix T\cap D}
\leq \theta\paren{d\paren{x^k,\Fix T\cap D}} 
\quad \forall k \in \mathbb{N},
\end{equation}%
where $\theta$ given implicitly by \eqref{eq:gauge} satisfies \eqref{eq:theta}. 
Moreover, the sequence $(x^k)_{k\in\Nbb}$
converges gauge monotonically  to 
some $x^{*}\in\Fix T\cap D$ with rate 
$O(s_k(t_0))$ where 
$s_k(t)\equiv 
\sum_{j=k}^\infty \theta^{(j)}(t)$ and $t_0\equiv d(x^0,\Fix T\cap D)$.
\end{enumerate}
 \end{thm}
Before proving this theorem, we collect some intermediate results. 

\begin{lemma}[gauge monotonicity and almost quasi $\alpha$-firmness implies convergence
to fixed points]
\label{t:rm and qafne to convergence}
Let $(G, d)$ be a $p$-uniformly convex space with constant $c$.  
Let $\mymap{T}{G}{G}$ 
with  $T(D)\subseteq D\subseteq G$ for $D$ closed.
Suppose that $\Fix T\cap D$ is nonempty and that $T$ is pointwise 
almost $\alpha$-firmly nonexpansive at all $y\in \Fix T\cap D$ with the same constant 
$\alphabar$ and violation $\epsilon$ on $D$. If the sequence $(x^k)_{k\in\mathbb{N}}$ defined by 
$x^{k+1}= Tx^k$ and initialized in $D$ 
is gauge monotone relative to $\Fix T\cap D$ 
with rate $\theta$ satisfying \eqref{eq:theta},
then $(x^k)_{k\in\mathbb{N}}$ converges gauge monotonically  to 
some $x^*\in\Fix T\cap D$ with rate $O(s_k(t_0))$ where 
$s_k(t)\equiv 
\sum_{j=k}^\infty \theta^{(j)}(t)$ and $t_0\equiv d(x^0,\Fix T\cap D)$.
\end{lemma}
\begin{proof}
The assumption that $T$ is pointwise almost $\alpha$-firmly nonexpansive at all
 $y\in\Fix T\cap D$ with constant $\alphabar$ and violation $\epsilon$ on $D$ yields 
 $$ d(Tx, y)^p\leq  (1+\epsilon)d(x, y)^p-\tfrac{c(1-\alphabar)}{2\alphabar}d(Tx,x)^p,
 \hspace{0.2cm}\forall x\in D.$$
 Let $x^0\in D$ and define the sequence $x^{k+1}= Tx^k$ for all $k\in\mathbb{N}$.
 Since $T$ is pointwise almost $\alpha$-firmly nonexpansive 
 at all points in $\Fix T\cap D$ on $D$,  $\Fix T\cap D$ is closed and  
 $P_{\Fix T\cap D}x^k$ is nonempty (though possibly set-valued) for all $k$.  
 Denote any selection by $\bar{x}^k\in P_{\Fix T\cap D}x^k$ for each $k\in\mathbb{N}$. 
 Then 
 $$d(x^{k+1}, \bar{x}^k)^p\leq  (1+\epsilon)d(x^k, \bar{x}^k)^p-
 \tfrac{c(1-\alphabar)}{2\alphabar}
d(x^k, x^{k+1})^p,\hspace{0.2cm}\forall k\in \mathbb{N},$$
 which implies  that
 $$d(x^k, x^{k+1})\leq  
 \paren{\tfrac{2\alphabar(1+\epsilon)}{c(1-\alphabar)}}^{1/p}d(x^k, \bar{x}^k),
 \hspace{0.2cm}\forall k\in \mathbb{N}.$$
 On the other hand 
 $d(x^k, \bar{x}^k)=
 d(x^k, \Fix T\cap D)
 \leq  \theta\paren{d(x^{k-1}, \Fix T\cap D)}$ 
 since $(x^k)_{k\in\mathbb{N}}$ is gauge monotone relative to $\Fix T\cap D$ 
 with rate $\theta$. Therefore an iterative application of gauge monotonicity yields
 $$d(x^k, x^{k+1})
 \leq   \paren{\tfrac{2\alphabar(1+\epsilon)}{c(1-\alphabar)}}^{1/p}\theta^{(k)}\paren{d(x^0, \Fix T\cap D)},
 \hspace{0.2cm}\forall k\in \mathbb{N}.$$
 Let $t_0=d(x^0, \Fix T\cap D)$. 
 For any given natural numbers $k,l$ with $k<l$ an iterative application of the triangle 
 inequality yields the upper estimate 
\begin{eqnarray*}
	d(x^k, x^l)&\leq& d(x^k, x^{k+1})+d(x^{k+1}, x^{k+2})+...+d(x^{l-1}, x^l)\\
 &\leq& \paren{\tfrac{2\alphabar(1+\epsilon)}{c(1-\alphabar)}}^{1/p}
 \paren{ \theta^{(k)}(t_0)+\theta^{(k+1)}(t_0)+\dots+\theta^{(l-1)}(t_0)}\\
 &<&  \paren{\tfrac{2\alphabar(1+\epsilon)}{c(1-\alphabar)}}^{1/p} s_{k}(t_0),
\end{eqnarray*}
where $s_k(t_0)\equiv \sum_{j=k}^{\infty}\theta^{(j)}(t_0)<\infty$ 
for $\theta$ satisfying \eqref{eq:theta}.  
Since $(\theta^{(k)}(t_0))_{k\in\Nbb}$ is 
a summable sequence of nonnegative numbers, the sequence of 
partial sums $s_{k}(t_0)$ converges to zero monotonically as $k\to\infty$
and hence $(x^k)_{k\in\mathbb{N}}$ 
is a Cauchy sequence and $x^k\to x^*$ for some 
$x^*\in G$.  Letting $l\to+\infty$ yields
$$\lim_{l\to+\infty}d(x^k, x^l)=d(x^k, x^*)\leq  a s_{k}(t_0),
\hspace{0.2cm}a\equiv\paren{\frac{2\alphabar(1+\epsilon)}{c(1-\alphabar)}}^{1/p}.$$
Therefore $(x^k)_{k\in\mathbb{N}}$ converges gauge monotonically to $x^*$ with rate $O(s_k(t_0))$.

It remains to show that $x^*\in\Fix T\cap D$. Note that for each $k\in \mathbb{N}$  
$$d(x^k, \bar{x}^k)=d(x^k, \Fix T\cap D)\leq  \theta^{(k)}(t_0),$$
which yields $\lim_k d(x^k, \bar{x}^k)=0$. 
But by the triangle inequality 
$$d(\bar{x}^k, x^*)\leq  d(x^k, \bar{x}^k)+d(x^k, x^*),$$
so $\lim_kd(\bar{x}^k, x^*)=0$. By construction 
$(\bar{x}^k)_{k\in\mathbb{N}}\subseteq \Fix T\cap D$ and  
$\Fix T\cap D$ is closed, hence $x^*\in\Fix T\cap D$. 
\end{proof}
 
\begin{proposition}[\cite{BLL}, Theorem 32]
\label{t:msr necessary}
Let $(G, d)$ be a $p$-uniformly convex metric space with constant $c$.  
Let $T:D\to D$ with  $D\subseteq G$.
Suppose that $S\equiv \Fix T\cap D$ is nonempty.  Suppose  
all sequences $(x^k)_{k\in\mathbb{N}}$ defined by $x^{k+1}=Tx^k$ 
and initialized in $D$ are gauge monotone 
relative to $S$ 
with rate $\theta$ satisfying \eqref{eq:theta}.  Suppose, in addition, that  
$(\Id - \theta)^{-1}(\cdot)$ is continuous on $\Rbb_+$, strictly increasing, 
and $(\Id - \theta)^{-1}(0)=0$.  Then $\Tcal_{S}$ defined by 
\eqref{eq:Tcal} is metrically subregular for $0$ relative to $D$ on $D$ 
with gauge $\mu(\cdot)=(\Id-\theta)^{-1}(\cdot)$.
\end{proposition}

 \noindent{\em Proof of Theorem \ref{t:msr convergence}. }
Part \eqref{t:msr convergence, necessary}.  This follows immediately from 
Lemma \ref{t:rm and qafne to convergence} and Proposition \ref{t:msr necessary}.  \\

Part \eqref{t:msr convergence, sufficient}.  Our pattern of proof follows the same logic as the 
analogous result for set-valued mappings in a Euclidean space setting \cite[Theorem 2.2]{LukTamTha18}.  
Since $S = \Fix T\cap D$, Proposition \ref{t:properties pafne}\eqref{t:properties pafne ii} establishes that 
$\psi(x,y)=\frac{c}{2}d(x, Tx)^p$ for all $y\in \Fix T$, so
$\Tcal_S(x)=\frac{c}{2}d(x, Tx)$.  Also by Proposition \ref{t:properties pafne}\eqref{t:properties pafne ii} 
$\Tcal_S$ takes the 
value $0$ only on $\Fix T$, that is, $\Tcal_S^{-1}(0)=\Fix{T}$.  
So by assumption that 
$\Tcal_S$ satisfies  \eqref{eq:estimate1} with gauge $\mu$ given by \eqref{eq:gauge} for 
 $\tau=(1-\alphabar)/\alphabar$, together with 
the definition of metric subregularity (Definition \ref{d:msr}) this yields
\begin{eqnarray}
d(x, \Fix T\cap D) &=& 
d(x, \Tcal_S^{-1}(0)\cap D)
  \nonumber\\
 &\le& 
\mu\paren{\Tcal_S(x)}
 = \mu(d(x, Tx))\quad\forall x\in D.\nonumber
\end{eqnarray}
In other words, 
\begin{equation}\label{eq:rate step 1}
   \tfrac{1-\alphabar}{\alphabar}
   \paren{\mu^{-1}\paren{%
   d(x,\Fix T\cap D)}}^{p}
\leq \tfrac{1-\alphabar}{\alphabar}d(x, Tx)^p
\quad\forall x\in D.
\end{equation}
On the other hand, by the assumption that 
$T$ is pointwise almost $\alpha$-firmly nonexpansive at all 
 points $y\in S$ 
 with the same constant $\alphabar$ and violation $\epsilon$ on $D$
 we have 
\begin{eqnarray}
0
  &\leq& \tfrac{1-\alphabar}{\alphabar}d(x, Tx)^p\nonumber\\
  &\leq&(1+\epsilon) d(x, y)^p-d(Tx, y)^p
      \quad\forall y\in  \Fix T\cap D , 
      \forall x\in D.
      \label{eq:rate step 2}
\end{eqnarray}
Incorporating \eqref{eq:rate step 1} into 
\eqref{eq:rate step 2} and rearranging the inequality yields
\begin{eqnarray}
d(Tx, y)^p
   \!&\leq&\! 
     (1+\epsilon)d(x, y)^p - 
         \tfrac{1-\alphabar}{\alphabar}
   \paren{\mu^{-1}\paren{%
    d(x, \Fix T\cap D)}}^{p}\nonumber\\
&&      \qquad\qquad\qquad\qquad\qquad
\qquad\forall y\in \Fix T\cap D , 
      \forall x\in D.\nonumber
\end{eqnarray}
Since this holds at {\em any} $x\in D$, it certainly 
holds at the iterates $x^k$ with initial point $x^0\in D$
since $T$ is a 
self-mapping on $D$.   Therefore for all $k\in \Nbb$
\begin{equation}\label{eq:gauge convergence 0}
d(x^{k+1}, y)
\leq \sqrt[p]{ (1+\epsilon)d(x^{k}, y)^p - 
\frac{1-\alphabar}{\alphabar}
\paren{\mu^{-1}\paren{%
 d\paren{x^{k},\, \Fix T\cap D}}}^p} 
\quad\forall y\in \Fix T\cap D.
\end{equation}%

Equation \eqref{eq:gauge convergence 0} simplifies.  
Indeed $\Fix T\cap D $ is closed, 
so for every $k\in\Nbb$  the distance  $d(x^k, \Fix T\cap D)$ 
is attained  at some $y^k\in \Fix T\cap D$ yielding
\begin{equation}
d(x^{k+1}, y^{k+1})^p \leq 
d(x^{k+1}, y^{k})^p \leq 
(1+\epsilon)d(x^k, y^k)^p  - 
\tfrac{1-\alphabar}{\alphabar}\paren{\mu^{-1}
	\paren{d(x^k, y^k)}}^p
\quad\forall k \in \mathbb{N}.
\label{eq:gauge convergence intermed}
\end{equation}
Taking the $p$th root and recalling \eqref{eq:gauge}
yields \eqref{eq:gauge convergence}.

This establishes also that the sequence $(x^k)_{k\in\Nbb}$ is 
gauge monotone relative to $\Fix T\cap D$ with rate 
$\theta$ satisfying 
Eq.\eqref{eq:theta}.  
By Lemma \ref{t:rm and qafne to convergence} it follows that 
the sequence $(x^k)_{k\in\Nbb}$ converges 
gauge monotonically to $x^*\in\Fix T\cap D$ with the rate
$O(s_k(d(x^0,\Fix T\cap D)))$ where $s_k(t)\equiv 
\sum_{j=k}^\infty \theta^{(j)}(t)$. 
\hfill $\Box$\\

\begin{corollary}[linear convergence]\label{t:linear msr convergence}
	\begin{enumerate}[(a)]
\item (necessity) In the setting of Theorem \ref{t:msr convergence}\eqref{t:msr convergence, necessary}, 
if all sequences $(x^k)_{k\in\mathbb{N}}$ defined by $x^{k+1}=Tx^k$ 
and initialized in $D$ are linearly monotone 
relative to $S$ 
with rate $\gamma <1$,
 then all sequences initialized on $D$ 
converge R-linearly to some $\xbar\in S$ with rate 
$O(\gamma^k)$.
Moreover, $\Tcal_{S}$ defined by 
\eqref{eq:Tcal} is linearly metrically subregular for $0$ relative to $D$ on $D$ 
with gauge $\mu(t)=(1-\gamma)^{-1}t$. 
		\item 	(sufficiency)  In the setting of Theorem \ref{t:msr convergence}\eqref{t:msr convergence, sufficient} 
		suppose that  $T$ satisfies 
		\begin{equation}
 \label{eq:estimate1-lin}
 d(x,\Fix T\cap D)\leq \mu d(x,Tx),\hspace{0.2cm}\forall x\in D,
 \end{equation}
with the scalar $\mu$ satisfying 
 \[
\sqrt[p]{\tfrac{1-\alphabar}{\alphabar(1+\epsilon)}}< \mu<  \sqrt[p]{\tfrac{1-\alphabar}{\alphabar\epsilon}}.
\]  
Then for any $x^0\in D$, the sequence $(x^k)_{k\in\Nbb}$ defined by 
$x^{k+1}= T x^k$ is $R$-linearly convergent to a point in $\Fix T\cap D$ with 
rate $\gamma=\sqrt[p]{1+\epsilon-\frac{1-\alphabar}{\alphabar\mu^p}}$.
	\end{enumerate}
\end{corollary}

In the statements above, the violation of $\alpha$-firm nonexpansiveness, $\epsilon$, has 
to be compensated for by an equally strong gauge of metric subregularity with this 
value of $\epsilon$ explicitly accounted for in the gauge.  
The next result shows that these can be decoupled if $T$ is pointwise asymptotically  
$\alpha$-firmly nonexpansive at fixed points.   In particular, if $T$ is pointwise almost $\alpha$-firmly 
nonexpansive at $y\in\Fix T$ with arbitrarily small violation $\epsilon$, then 
whenever $T$ is (gauge) metrically subregular at $y$, there is a neighborhood of 
$y$ on which convergence of the fixed point iteration can be quantified by said gauge.
In this situation it suffices to {\em qualitatively} determine metric subregularity  - the exact value of the 
constants in relation to the violation of $\alpha$-firmness is not needed in order to determine local 
convergence on the order of the gauge. 
\begin{proposition}\label{t:msr sufficient}
  Let $(G,d)$ be a $p$-uniformly convex space with constant $c$; let $D\subset G$, 
 let $\mymap{T}{D}{D}$,  
 and let $S \equiv\Fix T\cap D$ be nonempty. 
  Assume that $T$ is a self mapping on sufficiently small balls around points in 
  $S$ restricted to $D$, and 
 that $T$  is pointwise asymptotically $\alpha$-firmly nonexpansive at all 
 points $y\in S$ 
 with constant $\alphabar\in(0,1)$.  
 Suppose further that $T$ satisfies 
  \begin{equation}
 \label{eq:estimate1b}
 d(x,\Fix T\cap D)\leq \mu(d(x,Tx)),\hspace{0.2cm}\forall x\in D,
 \end{equation}
with gauge $\mu$ given by \eqref{eq:gauge} and 
 $\tau=(1-\alphabar)/\alphabar$.
Then for any $x^0$ close enough to $S$, the sequence $(x^k)_{k\in\Nbb}$ defined by 
$x^{k+1}= T x^k$ converges gauge monotonically  to 
some $x^{*}\in\Fix T\cap D$ with rate 
$O(s_k(t_0))$ where 
$s_k(t)\equiv 
\sum_{j=k}^\infty \theta^{(j)}(t)$ and $t_0\equiv d(x^0,\Fix T)$
for $\theta$ given implicitly by \eqref{eq:gauge} satisfying \eqref{eq:theta}. 
 \end{proposition}
 \begin{proof}
Since $T$ is a self mapping on $\Ball_\delta(S)\cap D$ for $\delta$ small enough, and 
$T$ is pointwise asymptotically $\alpha$-firmly nonexpansive with constant $\alphabar\in (0,1)$, the 
result follows immediately from Proposition \ref{t:persistence of msr} and Theorem \ref{t:msr convergence} 
	 when the domain $D$ is restricted to $\Ball_\delta(S)$ for $\delta$ sufficiently small.   
 \end{proof}

\section{Proximal Mappings}
We return now to the prox mapping \eqref{e:prox^p}. 
It was shown in \cite[Proposition 2.7]{Izuchukwu} that the $\argmin$ in \eqref{e:prox^p} exists and is 
unique if $f$ is proper, lsc and convex.  In general the prox mapping of a convex function is not $\alpha$-firmly 
nonexpansive.   However it was shown in \cite[Corollary 23]{BLL} that it is {\em almost}  
$\alpha$-firmly nonexpansive.  This and other properties of prox mappings are collected 
in the following result.

\begin{thm}\label{t:p-prox properties}
Let $(G,d)$ be a $p$-uniformly convex metric space with constant $c\in(0,2]$, 
and let $f\colon G \rightarrow \Rbb$ be proper, convex and lower semicontinuous. 
\begin{enumerate}[(i)]
	\item \label{t:quasi_ne_symetric} If $(G,d)$ is symmetric perpendicular, then
$\prox^p_{f,\lambda}$ is quasi strictly
nonexpansive, that is,
\[
d(\prox^p_{f,\lambda}(x),\xbar)< d(x,\bar{x})\quad \forall x \in G, \forall \xbar\in \argmin f = \Fix \prox^p_{f,\lambda}.   
\]

  \item\label{t:prox paafne} The prox mapping $ \prox_{f,\lambda}^p$ is pointwise almost $\alpha$-firmly nonexpansive
at all $y\in \argmin f$ on $G$ with constant 
\begin{equation}\label{e:alpha-epsilon_c}
	\alpha_c = \tfrac{c(c-1)}{c(c-1)+2}\quad\mbox{and violation }\quad \epsilon_c= \tfrac{2-c}{c-1}.
\end{equation}
\item\label{t:prox pasafne} If $(G,d)$ is a CAT$(\kappa)$ space, then $\prox_{f,\lambda}^p$ is  pointwise asymptotically $\alpha$-firmly nonexpansive 
at all $y\in \argmin f$ with constant $\overline{\alpha}= 1/2$.
\end{enumerate}
\end{thm}
\begin{proof}
	\eqref{t:quasi_ne_symetric}. Let $x \in G$ be arbitrary and $y \coloneqq \prox^p_{f,\lambda}(x)$. 
	We prove by contradiction that the projection of $x$ 
onto the geodesic $[\xbar,y]$ connecting $\xbar$ and $y$ is $y$, i.e. $P_{[\xbar,y]}(x)= y$. Therefore assume 
that $P_{[\xbar,y]}(x)\neq y$, i.e. $P_{[\xbar,y]}(x)= (1-t)y \oplus t \xbar$ for some $t \in (0,1]$. Then 
$f((1-t)y \oplus t \xbar) \leq (1-t) f(y) + t f(\xbar)\leq f(y)$ and $d((1-t)y \oplus t \xbar,x) < d(y,x)$. 
Now 
$$f((1-t)y \oplus t \xbar)+\frac{1}{p \lambda^{p-1}}d((1-t)y \oplus t \xbar,x)^p< f(y) +\frac{1}{p \lambda^{p-1}}d(y,x)$$
contradicts $y= \prox^p_{f,\lambda}(x)$.  Hence our assumption must be discarded and 
$P_{[\xbar,y]}(x) = y$. In particular $[\xbar,y] \perp_y [x,y]$ and hence by the symmetric perpendicular property 
$[x,y] \perp_y [\xbar,y]$. Now $[x,y] \perp_y [\xbar,y]$ in turn yields the claim $d(y,\xbar)\leq d(x,\xbar)$.

If in addition $(G,d)$ is a $p$-uniformly convex space
\begin{align*}
d(\xbar, y)^p&=d(\xbar,P_{[x,y]}(\xbar))^d \leq d(\xbar, \frac{1}{2} x \oplus \frac{1}{2} y)^p \\ 
&\leq \frac{1}{2} d(\xbar,x)^p + \frac{1}{2} d(\xbar, y)^p- \frac{c}{2} \frac{1}{4} d(x,y)^p\\
&\leq d(\xbar, x)^p - \frac{c}{2} \frac{1}{4} d(x,y)^p.
\end{align*}
This is only possible if either $x=y$ or $d(\xbar,y) < d(\xbar,x)$. In both cases $\prox^p_{f,\lambda}$ is quasi strictly nonexpansive.

\eqref{t:prox paafne}.  This is \cite[Corollary 23]{BLL}. \\

\eqref{t:prox pasafne}. Let $\epsilon >0$, $y\in \argmin f$ and $c=\frac{2+\epsilon}{1+\epsilon}$. 
Then $c \in (1,2)$ and there is a unique solution $t \in (0, \pi/2)$ to $\frac{c}{2} = t \tan(\pi/2-t) $. 
Set $\delta=t/(2 \sqrt{\kappa})$. Then $\delta \in (0,\pi/(4 \sqrt{\kappa}))$ and 
$(\Ball_\delta(y), \left.d\right|_{\Ball_\delta(y)})$ is a $p$-uniformly convex space with 
constant $c$ by Lemma \ref{t:CATkappa-pucvx}. 
Part \eqref{t:quasi_ne_symetric} ensures that $\prox^p_{f,\lambda}$ 
is a self mapping on $\Ball_\delta(y)$ and hence 
\begin{align*}
\argmin_{z \in G} f(z)+ \frac{1}{p \lambda^{p-1}} d(z,x)^p =\argmin_{z \in \Ball_\delta(y)} f(y)+ 
\frac{1}{p \lambda^{p-1}} d(y,x)^p
\end{align*}
and we will use the same notation $\prox^p_{f,\lambda}$ for the operator restricted to the subspace 
$(\Ball_\delta(y), \left.d\right|_{\Ball_\delta(y)})$. 
The operator $\prox^p_{f,\lambda}$ is pointwise almost $\alpha$-firmly nonexpansive with constant 
$\alpha_c=\frac{c(c-1)}{c(c-1)+2}$ 
and violation 
$\epsilon_c =  \frac{2-c}{c-1} =\epsilon$ 
on $(\Ball_\delta(y), \left.d\right|_{\Ball_\delta(y)})$ by part \eqref{t:prox paafne}. 
Note that $\alpha_c\nearrow 1/2$ as $c\nearrow 2$ and by Propositon \ref{t:properties pafne}\eqref{t:properties pafne iv}
$\prox^p_{f,\lambda}$ is pointwise almost $\alpha$-firmly nonexpansive with constant 
$\alphabar=1/2$ for all $c\in(0,2)$.  Since $\epsilon_c\searrow 0$ as $c\nearrow 2$, this implies that 
$\prox^p_{f,\lambda}$ is 
pointwise asymptotically $\alpha$-firmly nonexpansive at $y\in\argmin f$ with constant $\alphabar = 1/2$ 
and neighborhood  $\Ball_\delta(y)$ of $y$.
\end{proof}

\begin{remark}
	In part \eqref{t:quasi_ne_symetric} of Theorem \ref{t:p-prox properties}, quasi nonexpansiveness comes
	from the symmetric perpendicular property.  Strict quasi nonexpansiveness comes from the property of
	$p$-uniformly convex spaces.  
\end{remark}

The next result gathers properties of mappings built from prox mappings.  

\begin{proposition}\label{t:combproxfg 0}
Let $(G,d)$ be a $p$-uniformly convex space with constant $c\in (1,2]$ and let  
$f,g\colon G \rightarrow \Rbb\cup\{+\infty\}$ be convex, proper and lower semicontinuous.  
Assume that $\argmin f \cap \argmin g \neq \emptyset$.
\begin{enumerate}[(i)]
\item\label{t:combproxfg ii} The operator $T\coloneqq \prox_{f,\lambda_2}^p \circ \prox_{g,\lambda_1}^p$ 
($\lambda_1, \lambda_2>0$) is pointwise almost 
$\alpha$-firmly nonexpansive
at all points $y\in \Omega\equiv \argmin f\cap \argmin g$ on $G$ with 
constant $\alpha_\circ = \frac{2 (c-1)}{2c-1}$ and violation 
$\epsilon_\circ=\frac{1}{(c-1)^2}-1$.
If $(G, d)$ is a CAT($\kappa$) space, then $T$ is pointwise asymptotically 
$\alpha$-firmly nonexpansive at all $y\in\Omega$ with constant 
$\overline{\alpha}_\circ= 2/3$.   
\item\label{t:combproxfg iii0} Let $\mymap{T_0}{G}{G}$ be 
pointwise asymptotically $\alpha$-firmly nonexpansive with constant $\alphabar_0$ at all 
$y\in \Omega\equiv \argmin f\cap\Fix T_0$.  The operator 
$T\coloneqq \prox_{f,\lambda}^p \circ T_0$ 
($\lambda>0$) is pointwise almost 
$\alpha$-firmly nonexpansive
at all points $y\in \Omega$ on $G$ with constant and violation
 \begin{equation}\label{eq:alpha_0-c}
	 \alpha=\tfrac{\alphabar_{0}+\alpha_{c}-2\alphabar_{0}\alpha_{c}}%
	 {\frac{c}{2}\paren{1-\alphabar_{0}-\alpha_{c}+\alphabar_{0}\alpha_{c}}+
		 \alphabar_{0}+\alpha_{c}-2\alphabar_{0}\alpha_{c}}
	 \quad \mbox{and } \epsilon = \epsilon_0+\epsilon_c+\epsilon_0\epsilon_c
 \end{equation}
 where $\epsilon_0$ is the violation of $\alpha$-firm nonexpansiveness of $T_0$ on some neighborhood
 small enough. 
If $(G, d)$ is a CAT($\kappa$) space, then $T$ is pointwise asymptotically 
$\alpha$-firmly nonexpansive at all $y\in\Omega$ with constant 
\begin{equation}\label{eq:alphabar_0c}
	 \alphabar\equiv\tfrac{1}{2-\alphabar_0}.
\end{equation}

\item\label{t:combproxfg iii} The Krasnoselski-Mann relaxation $T\equiv \beta\prox_{f,\lambda}^p\oplus (1-\beta)\Id$ 
is pointwise almost $\alpha$-firmly nonexpansive
at all points $y\in \Omega\equiv \argmin f$ on $G$ with constant 
\[
\alpha_\beta = \frac{\alpha_c\beta^{p-1}}{\alpha_c\beta^{p-1} - \alpha_c\beta+1} 
\]
and violation 
$\epsilon_\beta=\beta\epsilon_c$ where $\alpha_c$ and $\epsilon_c$ are given by \eqref{e:alpha-epsilon_c}.
If $(G, d)$ is a CAT($\kappa$) space, then $T$ is pointwise asymptotically 
$\alpha$-firmly nonexpansive at all $y\in\Omega$ with constant 
\[
\overline{\alpha}_\beta\equiv \tfrac{\beta^{p-1}}{\beta^{p-1}-\beta+2}.
\]

\item\label{t:combproxfg iv} The composition 
$T\equiv \prox_{f,\lambda_2}^p\circ\paren{\beta\prox_{g,\lambda_1}^p\oplus (1-\beta)\Id}$ 
is pointwise almost $\alpha$-firmly nonexpansive
at all points $y\in \Omega\equiv \argmin f\cap \argmin g$ on $G$ with constant
 \begin{equation}\label{eq:alphabar2}
	 \alphahat=\tfrac{\alpha_{\beta}+\alpha_{c}-2\alpha_{\beta}\alpha_{c}}%
	 {\frac{c}{2}\paren{1-\alpha_{\beta}-\alpha_{c}+\alpha_{\beta}\alpha_{c}}+
		 \alpha_{\beta}+\alpha_{c}-2\alpha_{\beta}\alpha_{c}}
	 	 \quad \mbox{ and violation }\epsilonhat = (1+\beta)\epsilon_c + \beta\epsilon_c^2,
 \end{equation}
 where $\alpha_\beta$, $\alpha_c$, and $\epsilon_c$ are the constants in 
part \eqref{t:combproxfg iii} and Theorem \ref{t:p-prox properties}\eqref{t:prox paafne} respectively. 
If $(G, d)$ is a CAT($\kappa$) space, then $T$ is pointwise asymptotically 
$\alpha$-firmly nonexpansive at all $y\in\Omega$ with constant 
\[
\alphahat= \tfrac{1}{2 - \overline{\alpha}_\beta}.   
\]
where $\overline{\alpha}_\beta$ is the constant in part \eqref{t:combproxfg iii},

\item\label{t:combproxfg v}  If $(G, d)$ is symmetric perpendicular,  
the projected gradient operator 
\[
T\equiv P_C\circ\paren{\beta\prox_{g,\lambda}^p\oplus (1-\beta)\Id}                                 
\]
is pointwise almost $\alpha$-firmly nonexpansive
at all points $y\in \Omega\equiv C\cap\argmin g$ on $G$ with 
 \begin{equation}\label{eq:alphabar PG}
	 \alpha_{PG}=\tfrac{1}%
	 {\frac{c}{2}\paren{1-\alpha_{\beta}}+ 1}
	 \quad \mbox{ and violation }\epsilon_{PG} = \epsilon_c\beta.
 \end{equation}
If $(G, d)$ is a CAT($\kappa$) space, then $T$ is pointwise asymptotically 
$\alpha$-firmly nonexpansive at all $y\in\Omega$ with constant 
\[
\alpha_{PG}= \tfrac{1}{2 - \overline{\alpha}_\beta}. 
\]
where $\overline{\alpha}_\beta$ is the constant in part \eqref{t:combproxfg iii}, 
\end{enumerate}
\end{proposition}
\begin{proof}
	\eqref{t:combproxfg ii}.  
	By  Theorem \ref{t:p-prox properties}\eqref{t:prox paafne}, the operators $\prox_{f,\lambda}^p$ and 
	$\prox_{g,\lambda}^p$ are  almost 
	quasi $\alpha$-firmly nonexpansive with constants 
$\alpha_c=\frac{c (c-1)}{c (c-1)+2}$ and violation $\epsilon_c=\frac{2-c}{c-1}$. 
The operator $T$ is almost $\alpha$-firmly nonexpansive at all points $y\in \Omega$ on $G$ with 
\begin{align*}
\alpha_\circ= \frac{2 \alpha_c-2 \alpha_c^2}{\frac{c}{2} (1-2\alpha_c+\alpha_c^2)+2\alpha_c-2\alpha_c^2}
\end{align*}
and violations $1+\epsilon_\circ=(1+\epsilon_c)^2$ by Proposition \ref{t:compositionthm} .
A short calculation yields
\begin{align*}
	\alpha_\circ=\frac{2(c-1)}{2c-1}\quad\mbox{ and }\quad \epsilon_\circ=\frac{1}{(c-1)^2}-1.
\end{align*}
Taking the limit as $c\to 2$ from below yields the constant 
$\alphabar_\circ = 2/3$ with limiting violation $\epsilonbar_\circ = 0$.  The same argument 
as Propositon \ref{t:p-prox properties}\eqref{t:prox pasafne} then shows that, when $G, d)$ is a 
CAT$(\kappa)$ space,  the composition of 
two prox mappings is pointwise asymptotically $\alpha$-firmly nonexpansive 
at points in $\Omega$ with constant $2/3$.\\

\eqref{t:combproxfg iii0}.  Theorem \ref{t:p-prox properties}\eqref{t:prox paafne} and Proposition 
\ref{t:compositionthm} yield pointwise almost 
$\alpha$-firm nonexpansiveness of $T$ at $y\in\Omega$ with constant and violation given by \eqref{eq:alpha_0-c}.
 where $\alpha_c$ and $\epsilon_c$ are given by \eqref{e:alpha-epsilon_c},
 $\alphabar_0$ is the asymptotic constant of $\alpha$-firm nonexpansiveness of $T_0$,
and  $\epsilon_0$ is the violation on some neighborhood.
 (By Proposition \ref{t:properties pafne}\eqref{t:properties pafne iv} if $T_0$ is 
pointwise almost $\alpha$-firmly nonexpansive with constant $\alpha_0<\alphabar_0$, then 
$T_0$ is  also pointwise almost $\alpha$-firmly nonexpansive with constant $\alphabar_0$.) 
By  Theorem \ref{t:p-prox properties} \eqref{t:prox pasafne} and the assumption that 
$T_0$ is pointwise asymptotically $\alpha$-firmly nonexpansive at $\Fix T_0$ 
with constant $\alphabar_0$, 
the same argument as Proposition \ref{t:p-prox properties}\eqref{t:prox pasafne} 
establishes that, when $G, d)$ is a CAT$(\kappa)$ space,  $T$ is 
pointwise asymptotically $\alpha$-firmly nonexpansive at points in $\Omega$ 
with constant $\alphabar$ given by 
\eqref{eq:alphabar_0c}.  \\

\eqref{t:combproxfg iii}. The first statement is an immediate application of Proposition \ref{t:KM} and 
Theorem \ref{t:p-prox properties}\eqref{t:prox paafne}.  When $(G, d)$ is a CAT$(\kappa)$ space,
the same argument as Proposition \ref{t:p-prox properties}\eqref{t:prox pasafne} 
establishes that Krasnoselsky-Mann relaxations of prox mappings are  
pointwise asymptotically $\alpha$-firmly nonexpansive with constant $\alphabar_\beta$ as claimed. 
\\

\eqref{t:combproxfg iv}.  This is an application of part \eqref{t:combproxfg iii0} to part \eqref{t:combproxfg iii}.\\

\eqref{t:combproxfg v}.  This is a specialization of part \eqref{t:combproxfg iv} when $f$ is the indicator function of a convex set $C$
and follows from the fact that, on symmetric perpendicular $p$-uniformly convex spaces, 
the projector is pointwise $\alpha$-firmly nonexpansive at all points in $C$ 
with constant $\alpha=1/2$ (no violation) as shown in \cite[Proposition 25]{BLL}. \\
\end{proof}

\begin{remark}
	Part \eqref{t:combproxfg ii} of Proposition  \ref{t:combproxfg 0} coincides with 
	$\alpha = \frac{2}{3}$ and $\epsilon=0$ in the classic setting with $c=2$. 
	In particular the composition $P_A \circ P_B$ of two projections $P_A$ and $P_B$ onto convex sets 
	$A$ and $B$ with $A\cap B \neq \emptyset$ is $\alpha$-firmly nonexpansive at all $y\in A\cap B$ 
	on $G$ with $\alpha=\frac{2}{3}$ and violation $\epsilon=0$. 
	However this result does not apply if the problem is infeasible, i.e. $A \cap B= \emptyset$.
\end{remark}

These properties allow us to prove the following fundamental result. 
\begin{thm}[convergence of proximal algorithms in CAT($\kappa$) spaces]
	\label{t:convergence ppa}
	Let $(G, d)$ be a $CAT(\kappa)$ space with $\kappa >0$ and 
	$f,g\colon G \rightarrow \Rbb\cap \{+\infty\}$ be  
proper, convex and lower semicontinuous with 
$\argmin f\cap \argmin g  \neq \emptyset$. 
Let $T$ denote any of the mappings in Proposition \ref{t:combproxfg 0}. 
If  $T$ satisfies 
   \begin{equation}
 \label{eq:estimate1c}
 d(x,\Fix T\cap D)\leq \mu d(x,Tx),\hspace{0.2cm}\forall x\in D,
 \end{equation}
with constant  $\mu$, then the fixed point sequence  initialized from any starting point close 
enough to $\Fix T\cap D$ is at least linearly convergent 
to a point in $\Fix T\cap D$ with rate 
$\gamma = \sqrt{1+\epsilon - \tfrac{1-\alpha}{\alpha \mu^2}}<1$, where $\alpha$ and 
$\epsilon$ are the respective constant and violation of pointwise $\alpha$-firm nonexpansiveness
of the fixed point mapping $T$ as given in Proposition \ref{t:combproxfg 0}.  The asymptotic rate 
of convergence is $\gammabar = \sqrt{1- \tfrac{1-\alphabar}{\alphabar \mu^2}}<1$, where $\alphabar$
is  the respective constant of pointwise asymptotic $\alpha$-firm nonexpansiveness
of the fixed point mapping $T$.
\end{thm}
\begin{proof}
As established in Theorem \ref{t:p-prox properties} and Proposition \ref{t:combproxfg 0}, 
all of the mappings covered in those results are pointwise asymptotically $\alpha$-firmly 
nonexpansive at points in $\Omega$ with constants $\alphabar<1$ where 
$\Omega$ is one of the following subsets corresponding to the respective mappings (i)-(v)
 in \ref{t:combproxfg 0}: (i) $\Omega = \argmin f\cap \argmin g$; 
 (ii) $\Omega \subset \Fix T_0\cap \argmin f$; 
(iii) $\Omega = \argmin f$; (iv) $\Omega = \argmin f\cap\argmin g$; (v) $\Omega = C\cap\argmin g$.

As noted in Remark \ref{r:remsymmetric}, any CAT($\kappa$) space is symmetric perpendicular
locally, so by Theorem \ref{t:p-prox properties}\eqref{t:quasi_ne_symetric} and  
Lemma \ref{t:intersections}, in every case $\Fix T = \Omega$.  
 Now, if   $T$ satisfies \eqref{eq:estimate1c} at all 
points in $\Fix T$ on 
$G$ with constant $ \mu$, 
it follows immediately from Proposition \ref{t:msr sufficient} that 
the fixed point iteration  
converges linearly to a point in $\Fix T\cap D$  with the given rate
for all starting points close enough to $\Fix T\cap D$,  as claimed.  
\end{proof}

\section{Proximal Splitting Methods}
The concrete examples provided here have been well studied for $p$-uniformly convex spaces 
with $p=c=2$, i.e. CAT(0) spaces.  The tools established in the 
previous sections open the door to applying these methods in CAT$(\kappa)$ spaces, which is new. 
Since CAT$(\kappa)$ spaces are $p$-uniformly convex with $p=2$, to avoid confusion, we
revert to the usual notation for proximal operators in the setting, namely $\prox_{f,\lambda}$, 
omitting the exponent. 

Let $(G,d)$ be a CAT$(\kappa$) space,  $f_i:G\to G$ be proper lsc 
convex functions for $i=1,2,\dots N$.  Consider the problem
\begin{equation}
 \label{eq:sum of cvx}
 \inf_{x\in G}\sum_{i=1}^N f_i(x).
 \end{equation}
Applying {\em backward-backward splitting} to this problem yields 
 Algorithm \ref{alg:bbs}.
\begin{algorithm}[h]    
\SetKwInOut{Input}{Parameters}\SetKwInOut{Output}{Initialization}
  \Input{Functions $f_1\ldots,f_N$ and $\lambda_i >0$ $(i=1,2,\dots,N)$.}
  \Output{Choose  $x_0\in G$.}
  \For{$k = 0,1, 2, \ldots $}{
    \begin{align*}
		x_{k+1}=Tx_k\equiv 
		\paren{\prox_{f_N, \lambda_N}\circ\cdots\circ \prox_{f_2, \lambda_2}\circ \prox_{f_1, \lambda_1}}(x_k)
    \end{align*}
}
\caption{
Proximal splitting}\label{alg:bbs}
\end{algorithm}
Local linear convergence follows from Theorem \ref{t:convergence ppa} and the extension of 
Proposition \ref{t:combproxfg 0}\eqref{t:combproxfg ii} via part \eqref{t:combproxfg iii0}
of the same proposition and induction,  
under the assumption 
that $\Tcal_\Omega$ defined by \eqref{eq:Tcal} -- which  simplifies to \eqref{eq:Tcal_Fix_T} -- 
is linearly metrically subregular for $0$ on $G$ with 
constant $\mu$ and $\Omega\equiv \cap_{j}\argmin f_j\neq\emptyset$.  
By Proposition  \ref{t:properties pafne}\eqref{t:properties pafne ii}, linear metric subregularity 
simplifies to 
 \begin{equation}\label{eq:simple msr}
  d(x, \Omega)\leq \mu d(Tx, x)\quad \forall x\in D. 
\end{equation}

Recall, 
in a $p$-uniformly convex space with $p=2$, the Moreau-Yosida envelope of $f$ is defined by 
\[
e_{f, \lambda}(x)\equiv \inf_{y\in G}\paren{f(y)+\tfrac{1}{2\lambda}d(x,y)^2}. 
\]
The analogue to the direction of steepest descent for the Moreau-Yosida envelope in $p$-uniformly convex 
settings is
\begin{equation}\label{e:sd Me}
	(1-\beta)x\oplus \beta \prox_{f,\lambda}(x).
\end{equation}
Specializing problem \eqref{eq:sum of cvx} to the case $N=2$ and 
$f_2=\iota_C$, the indicator function of some closed convex set $C\subset G$ yields 
Algorithm \ref{alg:pg},  the analog to projected gradients in  CAT($\kappa$) space, 
which is the projected resolvent/projected prox iteration. 
\begin{algorithm}[h]    
\SetKwInOut{Input}{Parameters}\SetKwInOut{Output}{Initialization}
\Input{$\mymap{f}{G}{\Rbb}$, the closed set $C\subset G$, $\lambda>0$ and $\beta\in(0,1)$.}
  \Output{Choose  $x_0\in G$.}
  \For{$k = 0,1, 2, \ldots $}{
    \begin{align*}
		& x_{k+1}=T_{PG}(x_k)\equiv P_{C}\paren{(1-\beta)\Id\oplus \beta \prox_{f,\lambda}}(x_{k})
    \end{align*}
}
\caption{Metric Projected Gradients}\label{alg:pg}
\end{algorithm}
Local linear convergence follows immediately from Theorem \ref{t:convergence ppa} 
under the assumption 
that $T$ satisfies \eqref{eq:estimate1-lin}
 and $ \Omega\equiv \argmin f\cap C\neq\emptyset$.  

Compositions of projectors in CAT$(\kappa)$ spaces has been studied in 
\cite{BLL} and  \cite{RuiLopNic15}.  We consider 
Algorithm \ref{alg:bbs} when the functions $f_i\equiv \iota_{C_i}$, the indicator functions 
of closed convex sets $C_i\subset G$, where $(G, d)$ is a 
complete, symmetric perpendicular $p$-uniformly convex space with constant $c$.  
The $p$-proximal mapping of the indicator  function is the metric projector and so by 
\cite[Proposition 25]{BLL} these are pointwise $\alpha$-firmly nonexpansive at all 
points in $\cap_i C_i$ (assuming, of course, that this is nonempty) and   
by \cite[Lemma 10]{BLL} the cyclic projections mapping
\begin{equation}\label{e:cp mapping}
	T_{CP}\equiv P_{C_N}\cdot P_{C_2} P_{C_1}
\end{equation}
is pointwise $\alpha$-firmly nonexpansive at all 
points in $\cap_i C_i=\Fix T_{CP}$, when the intersection is nonempty, 
with constant $\alphabar_N = \frac{N-1}{N}$
on $G$.  $\Delta$- or weak convergence (no rate) to a point in  $\cap_i C_i$ 
follows from \cite[Theorem 27]{BLL}, with 
strong convergence whenever one of the sets is compact. 
If in addition 
$	 d(x,\cap_i C_i)\leq \mu d(T_{CP}x, x)$
for all  $x\in G$  where $\mu>0$ is the rate of linear metric subregularity, 
 then, by Theorem \ref{t:convergence ppa} below, the sequence $(x_k)$ initialized
 anywhere in $G$ converges 
 linearly to some  $x^{*}\in\cap_i C_i$.

\section{Open Problems}
There are two obvious next steps for this work.  First and foremost is to determine the 
requirements for quantitative convergence of proximal splitting methods for the 
case when the individual prox mappings do not have common fixed points -- the so-called
{\em inconsistent case} -- since it is too limiting to require that the fixed points of the 
constituent elements of splitting methods coincide.  The second item to explore is qualitative settings in which 
metric subregularity comes ``for free''.  In linear settings, polyhedrality and isolated
fixed points suffice to guarantee metric subregularity \cite[Propositions  3I.1 and 3I.2]{DontchevRockafellar14} and this was 
successfully used to prove local linear convergence of the ADMM/Douglas-Rachford algorithms in a convex setting 
\cite[Theorem 2.7]{ACL16}.  In more general settings, 
the Kurdyka-{\L}ojasiewicz (KL) property -- which is equivalent to metric subregularity \cite[Corollary 4 and Remark 5]{BolDan2010}
-- is  satisfied by {\em semi-algebraic} functions \cite{BolteDanilidisLewis2006}.
 Analogues to these properties 
for $p$-uniformly convex spaces would be very useful.  

\section{Declarations}
\subsection{Availability of data and materials} Not applicable
\subsection{Competing interests} The authors declare that they have no competing interests
\subsection{Funding} FL was  supported in part by 
the Deutsche Forschungsgemeinschaft (DFG, German Research Foundation) – 
Project-ID LU 1702/1-1. DRL was supported in part by 
Deutsche Forschungsgemeinschaft (DFG, German Research Foundation) – 
Project-ID LU 1702/1-1 and in part by Deutsche Forschungsgemeinschaft (DFG, German Research Foundation) – 
Project-ID 432680300 – SFB 1456.
\subsection{Authors' contributions}  The authors contributed equally to all results. 
All authors read and approved the final manuscript.
\subsection{Acknowledgements} Not applicable. 


\end{document}